\documentclass[11pt]{amsart}
\usepackage{amssymb}
\usepackage[english]{babel}
\usepackage{float}
\usepackage{booktabs}   
\usepackage{latexsym}
\usepackage{amsmath, amsfonts}
\usepackage{bbm}
\usepackage{amsthm,amssymb,fancyhdr,verbatim,graphicx}
\usepackage[usenames,dvipsnames]{color}
\definecolor{darkred}{rgb}{0.4,0.1,0.1}
\definecolor{darkblue}{rgb}{0.1,0.1,0.4}
\usepackage[colorlinks=true,linkcolor =darkred, citecolor=darkblue]{hyperref}
\theoremstyle{plain}
\newtheorem{hyp}{Hypothesis}[section]
\newtheorem{thm}{Theorem}[section]

\newtheorem{example}{Example}[section]

\newtheorem*{thm*}{Theorem}

\newtheorem{lem}[thm]{Lemma}

\newtheorem{prop}[thm]{Proposition}
\newtheorem{cor}[thm]{Corollary}

\newtheorem{dfn}[thm]{Definition}
\theoremstyle{remark}
\newtheorem{remark}[thm]{Remark}
\theoremstyle{plain}

\newcommand{\be}{\begin{equation}}
\newcommand{\ee}{\end{equation}}
\newcommand{\beu}{\begin{equation*}}
\newcommand{\eeu}{\end{equation*}}
\newcommand{\besu}{\begin{equation*}
\begin{aligned}}
\newcommand{\eesu}{\end{aligned}
\end{equation*}}
\newcommand{\bes}{\begin{equation}
\begin{aligned}}

\newcommand{\ees}{\end{aligned}
\end{equation}}

\newcommand\cB{\mathcal B}

\newcommand\cF{\mathcal F}

\newcommand\cH{\mathcal H}

\newcommand\cM{\mathcal M}

\newcommand\cR{\mathcal R}

\newcommand\cS{\mathcal S}

\renewcommand\frq{\mathfrak q}

\newcommand\ov{\overline}

\newcommand\wt{\widetilde}

\newcommand\wh{\widehat}

\newcommand\sess{\sigma_{\rm ess}}

\newcommand\void[1]{}

\def\sess{\sigma_{\rm ess}}

\def\ran{{\rm ran}\,}

\def\frt{{\mathfrak t}}

      \def\dC{{\mathbb C}}

   \def\dN{{\mathbb N}}   

      \def\dR{{\mathbb R}}

   \def\dZ{{\mathbb Z}}

   \def\cB{{\mathcal B}}   

      \def\cF{{\mathcal F}}

   \def\cH{{\mathcal H}}

\def\cM{{\mathcal M}}      

   \def\cQ{{\mathcal Q}}   \def\cR{{\mathcal R}}

\def\cS{{\mathcal S}}

\newcommand{\dom}{\mathrm{dom}\,}

\numberwithin{equation}{section}

\title[
2D Schr\"odinger operators perturbed by finite measures ]{Weakly
coupled bound state of 2D Schr\"odinger operator with
potential-measure}

\author{Sylwia Kondej and  Vladimir Lotoreichik}

\begin{document}

\maketitle



\begin{abstract}
We consider a self-adjoint two-dimensional Schr\"odinger operator
$H_{\alpha\mu}$, which corresponds to the formal differential
expression
\[
-\Delta - \alpha\mu,
\]
where $\mu$ is a finite compactly supported positive Radon measure
on $\dR^2$ from the generalized Kato class and $\alpha >0$ is the
coupling constant. It  was proven earlier that $\sigma_{\rm
ess}(H_{\alpha\mu}) = [0,+\infty)$. We show that for sufficiently
small $\alpha$ the condition $\sharp\sigma_{\rm d}(H_{\alpha\mu})
= 1$ holds and that the corresponding unique eigenvalue has the
asymptotic expansion
\[
\lambda(\alpha) = -(C_\mu +
o(1))\exp\Big(-\tfrac{4\pi}{\alpha\mu(\dR^2)}\Big), \qquad
\alpha\rightarrow 0+,
\]
with a certain constant $C_\mu > 0$. We obtain also the formula
for the computation of $C_\mu$.  The asymptotic expansion of the
corresponding eigenfunction is provided. The statements of this
paper extend Simon's results, see \cite{Si76}, to the case of
potentials-measures. Also for regular potentials our results are
partially new.
\end{abstract}


\section{Introduction}

Let  us consider a non-relativistic quantum particle living in a
two-dimen\-sional system and moving under the influence of the
potential $V\colon \dR^2\rightarrow \dR$  such that there exists
$\delta
>0$ for which
\begin{equation}
\label{a1} \int_{\dR^2} |V(x)|^{1+\delta} <
\infty\quad\text{and}\quad\int_{\dR^2}|V(x)|(1+|x|^\delta) <
\infty\,.
\end{equation}
The operator
$$
H_{\alpha V}=-\Delta - \alpha V\,:\, \dom H_{\alpha V}  \to  L^2
(\dR^2)
$$
is self-adjoint with $\dom H_{\alpha V} =H^2 (\dR^2)$ and it
determines the Hamiltonian of our system. This operator represents
the sesquilinear form
\[
\frt_{\alpha V}[f,g] := (\nabla f,\nabla g)_{L^2(\dR^2;\dC^2)} -
\alpha (V f,g)_{L^2(\dR^2)},\qquad \dom \frt_{\alpha V} :=
H^1(\dR^2).
\]
The spectrum of $H_{\alpha V}$ can not be computed explicitly for
an arbitrary potential. For this reason spectral estimates and
asymptotic expansions of spectral quantities related  to
$H_{\alpha V}$ attract a lot of attention. Weak coupling
asymptotic regime belongs to this line of research.  It was shown
by Simon in \cite{Si76} that under the assumptions
\begin{equation}\label{eq-potetial}
\int_{\dR^2} V(x) \ge 0\quad\text{and} \quad V\neq 0
\end{equation}
the operator $H_{\alpha V}$ has at least one bound state for any
 $\alpha >0$; moreover for $\alpha $ small the corresponding lowest
eigenvalue asymptotically behaves as
$$
\lambda(\alpha )\sim -\exp \Big(\Big[ \frac{\alpha }{4\pi}
\int_{\dR^2} V(x)\Big] ^{-1} \Big),
\qquad \alpha\rightarrow 0+,
$$
provided inequality \eqref{eq-potetial} is sharp;
~cf.~\cite[Theorem 3.4]{Si76}.

The problem  we study in this paper, is addressed in a certain
respect to a more general class of potentials which, for example,
includes so-called \emph{singular interactions}. To sketch the
physical context suppose that a particle in confined by a quantum
wire with  possibility of tunnelling. Consequently, the whole
space $\dR^2$ is available for the particle. On the other hand, if
the wire is very thin we can make an idealization and assume that
the particle is localized in the vicinity of the set $\Sigma
\subset \dR^2$ of a lower dimension. The Hamiltonian of such a
system can be formally written as
$$
-\Delta -\alpha\, \delta_\Sigma\,,\qquad \alpha >0 \,,
$$
where $\delta_\Sigma $ denotes  the Dirac measure supported on
$\Sigma$, see \cite{E08} for the review on such Hamiltonians. More
generally, one can speak of
$$
-\Delta -\alpha\, \mu \,, \qquad \alpha >0 \,,
$$
where $\mu $ is a positive finite Radon measure on $\dR^2$. In
order to give a mathematical meaning to the above formal
expression we assume that $\mu $ belongs to the generalized Kato
class as in Definition~\ref{dfn:Kato}. Under this assumption the
embedding of $H^1 (\dR^2 )$ into $L^2 (\dR^2;d\mu)$ is well
defined and the following closed, densely defined, symmetric and
lower-semibounded  sesquilinear form
\[
\frt_{\alpha\mu}[f,g] := (\nabla f,\nabla g)_{L^2(\dR^2;\dC^2)} -
\alpha \int_{\dR^2} f(x)\ov{g(x)}d\mu(x),\quad \dom \frt_{\alpha
\mu} := H^1(\dR^2),
\]
induces the uniquely defined self-adjoint operator $H_{\alpha
\mu}$ in $L^2(\dR^2)$. It is known that $\sigma_{\rm
ess}(H_{\alpha\mu}) = [0,+\infty)$, see \cite[Theorem
3.1]{BEKS94}. The following theorem contains all the main results
of the paper.
\begin{thm*}
Let $\mu $ be a compactly supported positive finite Radon measure
on $\dR^2$ from the generalized Kato class and $H_{\alpha\mu}$ be
the self-adjoint operator defined above. Then the following
statements hold.
\begin{itemize}\setlength{\parskip}{0.2cm}
\item [\rm (i)] For $\alpha > 0$ sufficiently small we have
\[
\sharp\sigma_{\rm d}(H_{\alpha\mu})  = 1\,.
\]
Denote this unique eigenvalue by $\lambda(\alpha) < 0$ and the
corresponding eigenfunction by $f_\alpha\in L^2(\dR^2)$.
\item [\rm (ii)] The asymptotic expansion of $\lambda (\alpha )$
takes  the form
\begin{equation*}
\label{asymp1} \lambda(\alpha) =  -(C_\mu +
o(1))\exp\big(-\tfrac{4\pi}{\alpha \mu(\dR^2)}\big), \qquad\alpha
\rightarrow 0+\,,
\end{equation*}
where  the constant $C_\mu$ is given in \eqref{Cmu}.
\item [\rm (iii)] Set $k_\alpha = (-\lambda (\alpha ))^{1/2}$.
Then the corresponding eigenfunction admits the following
expansion
\[
f_\alpha(\cdot) = \frac{k_\alpha}{2\pi}  \int_{\dR^2 }K_0 (
k_\alpha |\cdot-y| )d\mu (y) +O \Big( \frac{1}{\ln k_\alpha
}\Big),\qquad \alpha\rightarrow 0+,
\]
where $K_0(\cdot)$ is the Macdonald function, the norm of the
first summand has non-zero finite limit as $\alpha\rightarrow 0+$,
and the error term is understood in the strong sense.
\end{itemize}
\end{thm*}

The reader may note that in the asymptotic expansion of $\lambda
(\alpha ) $ the dominating  term depends only on the total measure
of $\dR^2$ and does not depend on the distribution character of
the measure $\mu$. This stays in consistency with Simon's result
and  reflects the property that in the weak coupling regime
spectral quantities ``forget" about local properties of the
potential.

The statements  of this paper constitute the  extensions and
generalizations of the results obtained in \cite{Si76}. Firstly,
the class of perturbations that we admit  contains, for example,
singular measures as $\delta$-distributions supported on sets of
lower dimensions. Secondly, for  regular compactly supported
potentials our class is slightly larger than that of \cite{Si76}. In order
to give the reader an idea of that, let us only mention that
radially symmetric potential
\[
V(r) = \frac{\chi(r)}{r^2|\ln(r)|^{\gamma
}}\,,
\]
 with $\chi(r)$ being the characteristic function of the
interval $[0,1/2]$ and  $\gamma > 2$, is compactly supported and
belongs to the generalized Kato class, however it does not satisfy
assumptions \eqref{a1}, which are imposed in \cite{Si76}. One
should say that the formula for the constant $C_\mu$ given in
\eqref{Cmu} is derived formally by physicists \cite{Pa80} in the
case of regular potentials, but without a rigorous mathematical
proof.

Analogous asymptotic expansions of the bound state with respect to
a small parameter appear  in various spectral problems. It is
worth to mention such results for two-dimensional waveguides with
weak local perturbations \cite{BGRS97} as well as for coupled
waveguides with a small window \cite{P99} and also with a
semi-transparent window \cite{EKr01}. Recently a ``leaky
waveguide'' with a small parameter breaking the symmetry was
considered in \cite{KK13}. For the similar problems in the
one-dimensional case see~\cite{BGS77, Kl77, LL58, Si76}. The
analogous results for quantum graphs were obtained in \cite{EEK10,
E96, K07}. See also recent developments for Pauli operators
\cite{FMV11}. Our list of references is far from being complete,
however many of significant related works are mentioned.

In order to prove the main statements we will apply the Birman-Schwinger
principle. Precisely saying, we use its generalization for
potentials-measures from the generalized Kato class, which is
rigorously established in \cite{BEKS94}, see also \cite{Br95} and
\cite{P01, BLL13} for further modifications. We also use some
simple results of perturbation theory of linear operators, where
the standard reference is \cite{Kato}, however we require some
extensions of the classical results.

The paper is organized as follows. In
Section~\ref{sec-preliminaries} we complete some mathematical
tools useful for further spectral analysis. Namely, we provide a
rigorous definition of the Hamiltonian $H_{\alpha  \mu}$,
formulate the Birman-Schwinger principle, develop a perturbation
method for a particular class of non-analytic operator families
and analyze the properties of the operators involved into the
Birman-Schwinger principle. In Section~\ref{sec-weakly} we
formulate and prove  main results of the paper concerning the
uniqueness of the bound state in the weak coupling regime, obtain
its asymptotic behavior and  derive the behavior of the
corresponding eigenfunction.

In the remaining part of the paper we employ the following
abbreviations:
\begin{itemize}\setlength{\parskip}{2mm}
\item[$\bullet$] we set $L^2 := L^2(\dR^2)$ (norm $\|\cdot\|$),
$L^1 := L^1(\dR^2)$, $H^k\equiv H^k (\dR^2)$ with $k\in\dZ$ (norm $\|\cdot\|_{k}$)  and
$\mathbb{L}^2 := L^2 (\dR^2; \dC^2 ) $;
\item[$\bullet$] the notation  $\mathcal{S} := \mathcal{S}(\dR^2)$
stands for the Schwartz class, moreover we set $\mathcal{S}' :=
\mathcal{S}'(\dR^2)$ for the space dual to $\mathcal{S}$, i.e.
$\mathcal{S}'$ is the space of linear continuous functionals on
$\mathcal{S}$;
\item[$\bullet$] we set $L^2_\mu := L^2 (\dR^2 ; d\mu )$ and $L^1_\mu :=
L^1 (\dR^2 ; d\mu )$;
\item[$\bullet$] for the positive Radon measure $\mu$ on $\dR^2$ we denote $\mu_{\rm T} : =\mu(\dR^2 )$.
\end{itemize}

\subsection*{Acknowledgements}
V.\,L. gratefully acknowledges financial support by the Austrian Science Fund (FWF), project P 25162-N26.
\section{Preliminaries} \label{sec-preliminaries}

This section plays an auxiliary role and consists of four
subsections. In Subsection~\ref{ssec:prelim1} we provide necessary
facts from \cite{BEKS94, Br95}  on self-adjoint free Laplacians
perturbed by Kato-class measures. In Subsection~\ref{ssec:prelim2}
we prove some statements on non-analytic perturbation theory,
which are hard to find in the  literature. In
Subsections~\ref{ssec:prelim3} and \ref{ssec:prelim4} we
complement known results on the operators related to
Birman-Schwinger principle.
\subsection{Self-adjoint Laplacians perturbed by Kato-class measures}
\label{ssec:prelim1} We start with  recalling the definition of
the generalized Kato class of positive Radon measures on $\dR^2$.
\begin{dfn}
\label{dfn:Kato} A positive Radon measure $\mu$ on $\dR^2$ belongs
to the generalized Kato class if
\[
\lim\limits_{\varepsilon\rightarrow
0+}\sup_{x\in\dR^2}\int_{D_\varepsilon(x)}
\big|\ln|x-y|\big|d\mu(y) = 0,
\]
where $D_\varepsilon(x)$ is the disc of radius $\varepsilon > 0$
with the center at $x\in\dR^2$.
\end{dfn}
\noindent Let $\mu$ be a positive Radon measure from the
generalized Kato class. Then for arbitrarily small $\varepsilon
>0$ there exists a constant $C(\varepsilon) > 0$ such that
\[
\int_{\dR^2}|f(x)|^2 d\mu(x) \le \varepsilon \|\nabla
f\|^2_{\mathbb{L}^2} + C(\varepsilon)\|f\|^2_{L^2}\,
\]
holds for every $f\in \cS$;
see \cite{BEKS94, SV96}. For the measure $\mu$ the embedding
operator $J_\mu \colon H^1 \rightarrow L^2_\mu $ is well-defined
as the closure of the natural embedding defined on the Schwartz
class, see \cite[Section 2]{BEKS94}. Consequently, the above
inequality has a natural extension, i.e.  for arbitrarily small
$\varepsilon
>0$ there exists a constant $C(\varepsilon) > 0$ such that
\begin{equation}
\label{J} \|J_\mu f\|^2_{L^2_\mu } \le \varepsilon \|\nabla
f\|^2_{\mathbb{L}^2 } + C(\varepsilon)\|f\|^2_{L^2},
\end{equation}
for all $f\in H^1$.
\begin{example}\rm{
Suppose that the measurable function $V\colon \dR^2\rightarrow
[0,+\infty)$ satisfies the condition
\[
\lim_{\varepsilon\rightarrow
0+}\sup_{x\in\dR^2}\int_{D_\varepsilon(x)}|\ln|x-y||V(y)dy = 0.
\]
Then the measure
\[
\mu_V(\Omega) := \int_{\Omega} V(x)dx
\]
belongs to the generalized Kato class.}
\end{example}
\begin{example}\cite[Example 2.3 (c)]{V09}\rm{
Given a family $\{\Gamma_i\}_{i=1}^N$ of Lipschitz curves in the
plane. Suppose that each curve in the family is parameterized by
its arc length $\gamma_i\colon [0,|\Gamma_i|]\rightarrow \dR^2$
and $\gamma_i([0,|\Gamma_i|]) = \Gamma_i$ with $i=1,2,\dots, N$.
Assume that there exist $c\in(0,1]$ such that for all $s,t\in [0,
|\Gamma_i|]$ the condition $|\gamma_i(s) - \gamma_i(t)| \ge 1/2
|s-t|$ holds with $i=1,2,\dots, N$. So that each curve can not
have cusps and can not intersect itself, whereas different curves
can intersect each other.
Now let
$\Gamma := \cup_{i=1}^N \Gamma_i$. Then the Dirac measure
supported on $\Gamma$ belongs to the generalized Kato class.}
\end{example}

Let the self-adjoint operator
\[
-\Delta \,:\, \dom(-\Delta )\rightarrow  L^2\,,\qquad \dom(-\Delta
)= H^2 \,,
\]
define the unperturbed Hamiltonian of our system. In fact,
$-\Delta $ represents closed, densely defined, symmetric and
lower-semibounded
 sesquilinear form
\begin{equation}
\label{frhfee} \frt[f,g]= (\nabla f,\nabla g)_{\mathbb{L}^2},\quad
\dom\frt = H^1.
\end{equation}
Let $\mu$ be a  positive Radon measure from the generalized Kato
class. By means of $\mu$ we define the sesquilinear form
\begin{equation}
\label{frhmu} \frt_{\alpha\mu}[f,g] := \frt[f,g] - (\alpha J_\mu
f, J_\mu g)_{L^2_\mu},\quad \dom\frt_{\alpha\mu} := H^1,
\end{equation}
which, in view of  \eqref{J} and
KLMN-theorem,~cf.~\cite[Theorem~X.17]{RS2}, is symmetric, closed
and lower-semibounded.

\begin{dfn}
\label{dfn:operator} Let $H_{\alpha\mu}$ be a self-adjoint
operator acting in $L^2$  and defined as the operator associated
with $\frt_{\alpha\mu}$ via  the first representation
theorem,~\cite[Chapter VI, Theorem~2.1]{Kato}.
\end{dfn}
Denote $R(\lambda) := (-\Delta - \lambda)^{-1}$ with
$\lambda\in\dC\setminus\dR_+$. Then $R(\lambda )$ is an integral
operator with the kernel
\[
G (x,y;\lambda) = \frac{1}{2\pi} K_0(i\sqrt{\lambda}|x-y|),\quad
x,y\in\dR^2,
\]
where $K_0(\cdot)$ is the Macdonald function, see \cite[\S
9.6]{AS64}. Following the notations of \cite{BEKS94} we introduce
the integral operator
\begin{equation}
\label{Rmu} R_{\mu\, dx } (\lambda ) \colon L^2_\mu \rightarrow
L^2 \,,\quad R_{\mu\, dx }f :=
\int_{\dR^2}G(x,y;\lambda)f(y)d\mu(y),
\end{equation}
and define the ``bilateral" embedding of $R(\lambda )$ to
$L^2_\mu$ by
\begin{equation}
\label{Q} Q(\lambda) := J_\mu R_{\mu\, dx }(\lambda) \colon
L^2_\mu \rightarrow L^2_\mu \,.
\end{equation}
Note that
\begin{equation}
\label{Qkernel} Q(-k^2)f := \frac{1}{2\pi}\int_{\dR^2} K_0(k|\cdot
-y|)f(y)d\mu(y).
\end{equation}
The Birman-Schwinger principle takes the following form.
\begin{prop}\cite[Lemma 1]{Br95}, \cite{BEKS94}
\label{prop:BS} Let $R_{\mu\, dx } (\cdot)$, $Q(\cdot)$ and
$H_{\alpha\mu}$ be as above. For $\lambda \in\dR_-$ the mapping
\[
h\mapsto R_{\mu\, dx }(\lambda )h
\]
is a bijection from $\ker(I - \alpha Q(\lambda))$ onto
$\ker(H_{\alpha\mu} - \lambda)$, and
\[
\dim\ker(I - \alpha Q(\lambda)) = \dim\ker(H_{\alpha\mu} -
\lambda).
\]
\end{prop}

We will also use the fact that the essential spectrum is stable
under a perturbation of a finite measure.
\begin{prop}\cite[Theorem 3.1]{BEKS94}
Let $\mu$ be a positive Radon measure on $\dR^2$ from the
generalized Kato class. Assume that
 $\mu_{\rm T} <\infty$ and $H_{\alpha\mu}$ is as
 in Definition~\ref{dfn:operator}. Then
\[
\sess(H_{\alpha\mu} )= [0,+\infty)
\]
holds.
\end{prop}
\begin{remark}
Note that also more singular perturbations are  considered. For
example, $\delta$-interactions supported on curves in $\dR^3$, see
\cite{EK02, EK03, EK08, K12, P01},  and $\delta'$-interactions
supported on hypersurfaces, see \cite{BEL13, BLL13, EJ13}. These
perturbations do not belong to the  generalized Kato class and
therefore they require different approaches.
\end{remark}

\subsection{Elements of non-analytic perturbation theory}
\label{ssec:prelim2} Putting in mind later purposes we analyse a
family of self-adjoint operators $k\mapsto T(k) $, $k \in\dR_+$,
acting in a Hilbert space  $\cH$ and taking the form
\[
T(k) := T_0 + \tfrac{1}{\ln k}T_1 +
O\big(\tfrac{1}{\ln^2k}\big),\qquad k\rightarrow 0+,
\]
where  $T_0=\varphi (\cdot , \varphi  )$ with  $\varphi \in \cH$
being a normalized function, $T_1$ is a bounded self-adjoint
operator in $\cH$ and the error term is understood in the operator
norm  sense. The family $T(\cdot )$ is not analytic and
consequently we can not apply directly the results of
\cite[Chapters~II~and~VII]{Kato}. In the following theorem we
investigate the spectra and the eigenfunctions of $T(k)$ in the
limit $k\rightarrow 0+$.
\begin{thm}
\label{thm:perturbation} Let  $k\mapsto T(k)$ be defined as above.
For sufficiently small $k > 0$ the spectrum
$\sigma(T(k))\subset\dR$ of $T(k)$ consists of two disjoint
components $\sigma_0(k)$ and $\sigma_1(k)$.
\begin{itemize}
\item[\rm (i)]The part $\sigma_0(k)$ is located in the small neighborhood
of zero and its diameter can be estimated as
\[
{\rm diam\,}\sigma_0(k) \le \tfrac{1}{|\ln k|}\|T_1\| +
O\big(\tfrac{1}{\ln^2 k}\big),\qquad k\rightarrow 0+.
\]
\item [\rm (ii)] The part $\sigma_1(k)$ consists of exactly one eigenvalue
$\omega(k)$ of multiplicity one, which depends on $k$
continuously.
\item [\rm (iii)] The normalized eigenfunction $\varphi_k$
corresponding to the eigenvalue $\omega(k)$ has the following
expansion
\begin{equation} \label{eq-ef}
\varphi_k =\varphi + O\big(\tfrac{1}{\ln k}\big),\qquad
k\rightarrow 0+,
\end{equation}
in the norm of $\cH$.
\item [\rm (iv)] The eigenvalue $\omega(k)$ admits the asymptotics
\begin{equation}\label{eq-asymptev}
\omega(k) = 1 + \tfrac{1}{\ln k}(T_1\varphi,\varphi) +
O\big(\tfrac{1}{\ln^2k}\big),\qquad k\rightarrow 0+.
\end{equation}
\end{itemize}
\end{thm}
\begin{proof}
\noindent (i) Note that $\sigma(T_0) = \{0,1\}$ and that $\varphi$
is an eigenfunction of the operator $T_0$ corresponding to the
eigenvalue $1$. The separation of the spectra of $T(k)$ into two
parts $\sigma_0(k)$ and $\sigma_1(k)$ for sufficiently small $k
>0$ follows from \cite[Theorem V.4.10]{Kato}. The component
$\sigma_0(k)$ is located in the neighborhood of $0$ and the
component $\sigma_1(k)$ is located in the neighborhood of $1$.
Note that again by \cite[Theorem V.4.10]{Kato} the diameter of
$\sigma_0(k)$ satisfies
\[
{\rm diam\,}\sigma_0(k) \le \tfrac{1}{|\ln k|}\|T_1\| +
O\big(\tfrac{1}{\ln^2 k}\big),\qquad k\rightarrow 0+.
\]

\noindent (ii) Let $E_i(k)$, $i=0,1$, be the orthogonal projectors
onto the spectral subspaces of the operator $T(k)$ corresponding to $\sigma_i(k)$.  Then
$E_0(0) = I - T_0$ and $E_1(0) =T_0$ hold. Since $\|T(k)-T_0\|$
tends to $0$ for $k\to 0+$, relying on \cite[Theorem 3]{DD87} we
have $\dim\ran E_1(k) = 1$ for sufficiently small $k
>0$. Therefore $E_1(k) = \wt\varphi_k(\cdot,\wt\varphi_k)$, where
$\wt\varphi_k$ is the normalized eigenfunction corresponding to
the eigenvalue $\omega(k)$ of $T(k)$ with multiplicity one.
According to \cite[Theorem VIII.1.14]{Kato} the eigenvalue
$\omega(k)$ depends on $k$ continuously.

\noindent (iii) By \cite[Proposition 2.1]{KMM07}, see also
\cite{BDM83}, the estimate
\[
{\rm dist}(\sigma_0(k),\sigma_1(k))\|E_0(k)E_1(0)\| \le
\frac{\pi}{2}\|T(k) - T(0)\|
\]
holds, which  yields  the asymptotic property
\[
\| E_1(0) - E_1(k)E_1(0)\| = O\Big(\tfrac{1}{\ln k}\Big),\qquad k
\rightarrow 0+,
\]
where we have used $E_0(k) = I - E_1(k)$. The above expansion implies  the following
\begin{equation}
\label{estimate1} \|\varphi -\wt\varphi_k(\varphi,\wt\varphi_k)\|
= O\Big(\tfrac{1}{\ln k}\Big),\qquad k \rightarrow 0+.
\end{equation}
A straightforward calculation yields
\[
\begin{split}
\|\varphi -\wt\varphi_k(\varphi,\wt\varphi_k)\|^2 &=\big(\varphi -\wt\varphi_k(\varphi,\wt\varphi_k),\varphi -\wt\varphi_k(\varphi,\wt\varphi_k)\big)\\
& = 1 - (\wt\varphi_k,\varphi)(\varphi,\wt\varphi_k) - (\varphi,\wt\varphi_k)\ov{(\varphi,\wt\varphi_k)} + |(\varphi,\wt\varphi_k)|^2\\
&\qquad\qquad\qquad\qquad\qquad\qquad\qquad\qquad =  1 -
|(\wt\varphi_k,\varphi)|^2\,.
\end{split}
\]
Combining the above result with the estimate \eqref{estimate1} we
arrive at
\begin{equation}
\label{estimate1.5} 1 - |(\wt\varphi_k,\varphi)|^2 =
O\Big(\tfrac{1}{\ln^2 k}\Big),\qquad k \rightarrow 0+.
\end{equation}
Consequently, we obtain
\begin{equation}
\label{estimate2} 1 - |(\wt\varphi_k,\varphi)| =
O\big(\tfrac{1}{\ln^2 k}\big),\qquad k \rightarrow 0+.
\end{equation}
Suppose that $(r(k), \theta(k))$ determine the polar
representation of $(\wt\varphi_k,\varphi)$, i.e.
$(\wt\varphi_k,\varphi) = r(k)e^{i\theta(k)}$.
According to \eqref{estimate2} we claim that
\begin{equation}
\label{rk} r(k) = 1 + O\big(\tfrac{1}{\ln^2 k}\big),\qquad k
\rightarrow 0+.
\end{equation}
Since $\wt\varphi_k$ is the normalized eigenfunction of $T(k)$
corresponding to the eigenvalue $\omega(k)$ the function
\begin{equation}
\label{rotation} \varphi_k := e^{i\theta(k)}\wt\varphi_k
\end{equation}
 is as well.  Thence, by \eqref{estimate1} and \eqref{rk} we get
\begin{equation*}
\label{eigenfunction}
\begin{split}
\|\varphi - \varphi_k\| &= \|\varphi -
e^{i\theta(k)}\wt\varphi_k\| \le \|\varphi -
r(k)e^{i\theta(k)}\wt\varphi_k\|
+ \|r(k)e^{i\theta(k)}\wt\varphi_k - e^{i\theta(k)}\wt\varphi_k\|\\
& = \|\varphi - (\varphi,\wt\varphi_k)\wt\varphi_k\| + |r(k)-1| =
O\big(\tfrac{1}{\ln k}\big),\qquad k \rightarrow 0+,
\end{split}
\end{equation*}
which proves the expansion \eqref{eq-ef}.

\noindent (iv) Moreover, $\omega(k)\in \sigma_1(k)$ as an
eigenvalue of $T(k)$ with multiplicity one admits the
representation
\[
\omega(k) = \Big( T(k)\varphi_k, \varphi_k  \Big)= \Big(T_0
\varphi_k, \varphi_k  \Big) +\tfrac{1}{\ln k} \Big (T_1 \varphi_k,
\varphi_k  \Big)
 + O\Big(\tfrac{1}{\ln^2
k}\Big),\qquad k \rightarrow 0+.
\]
Applying \eqref{eq-ef} and the fact that $T_1$ is bounded, we get
\[
\omega(k)= |(\varphi,\varphi_k)|^2 + \tfrac{1}{\ln
k}(T_1\varphi,\varphi) + O\big(\tfrac{1}{\ln^2k}\big),\qquad k
\rightarrow 0+.
\]
Using \eqref{estimate1.5} and \eqref{rotation} we get the
asymptotics of $\omega
(\cdot) $ given in \eqref{eq-asymptev}.
\end{proof}
%

\subsection{Properties of the $Q(\cdot)$-function}
\label{ssec:prelim3} In this subsection we analyze the
operator-valued function $Q(\cdot)$ defined in  \eqref{Q}. Our aim
is to describe certain basic properties of $Q(\cdot)$ and to
derive its asymptotic expansion in the neighborhood of zero. The
following lemma provides the first auxiliary tool.
\begin{lem}
\label{lem:bounded} Let $\mu$ be a compactly supported positive
finite Radon measure on $\dR^2$ belonging to the generalized Kato
class and $C\in\dR$ be a constant. Then the integral operator
acting as
\[
R f := \int_{\dR^2}\big(-\ln|\cdot-y| +C\big)f(y)d\mu(y)
\]
is bounded in $L^2_\mu$.
\end{lem}
\begin{proof}
The operator $R$ can be decomposed into the sum of two integral
operators:
\[
R_1f = \int_{\dR^2}\big(-\ln|\cdot-y|\big)f(y)d\mu(y)\,,\qquad
R_2f :=C\int_{\dR^2}f(y)d\mu(y)\,.
\]
According to the definition of the generalized Kato class
(Definition~\ref{dfn:Kato}) for any constant $A > 0$ one can find
$\varepsilon >0$ such that for every $x_0\in{\rm supp}\,\mu$ the
estimate
\begin{equation}
\int_{D_\varepsilon(x_0)}\big|\ln|x_0-y|\big|d\mu(y) \le A
\end{equation}
holds. Hence for any $x_0\in{\rm supp}\,\mu$ we get
\[
\begin{split}
&\int_{\dR^2}\big|\ln|x_0-y|\big|d\mu(y) \\
&\qquad = \int_{D_\varepsilon(x_0)}\big|\ln|x_0-y|\big|d\mu(y)+\int_{\dR^2\setminus\ov{D_\varepsilon(x_0)}}\big|\ln|x_0-y|\big|d\mu(y)\\
&\qquad\qquad\qquad\qquad\qquad\qquad\qquad\le A +
\max\big\{\big|\ln|\varepsilon|\big|,\big|\ln|{\rm diam}\,{\rm
supp}\,\mu|\big|\big\}\mu_{\rm T}.
\end{split}
\]
Note that the bound above is independent of the choice of $x_0$
and therefore by the Schur criterion \cite[Lemma 0.32]{Teschl} and
the symmetry of  the integral kernel the operator $R_1$ is
bounded. Let ${\mathbbm 1}_\mu$ stand for the identity function
from $L^2_\mu$. Note that the integral operator $R_2$ is a
rank-one operator $C{\mathbbm 1}_\mu(\cdot,{\mathbbm
1}_\mu)_{L^2_\mu}$. Consequently,  $R_2$ is also bounded. Now
boundedness of $R$ follows from decomposition $R= R_1 +R_2$ and
boundedness of $R_1$ and $R_2$ separately.
\end{proof}
After these preliminaries we are ready to analyze the
operator-valued function $\dR_+ \ni k\mapsto Q(-k^2)$. First, let
us note that for a given $k$ the operator $Q(-k^2)$ is bounded in
$L^2_\mu$. The proof of this fact can be done via repeating the
argument from \cite[Corollary~2.2]{BEKS94}. Now our aim is to
expand $Q(\cdot )$ in a neighbourhood of zero.
\begin{prop}
\label{prop:expansion} \label{prop-decomposition} Let $\mu$ be a
compactly supported positive Radon measure on $\dR^2$ from the
generalized Kato class,  and the operator-valued function
$Q(\cdot)$ be defined as in \eqref{Q}. Then $Q(\cdot)$ admits the
expansion
\begin{equation}\label{eq-Qdecomposition}
Q(-k^2) = -\ln(k)P + R + O(k^2\ln(k)),\qquad ~k\rightarrow 0+\,,
\end{equation}
in the operator norm, where  $P$ is a rank-one operator given by
\begin{equation}
\label{P} P  := \frac{1}{2\pi}{\mathbbm 1}_\mu\big(\cdot,{\mathbbm
1}_\mu\big)_{L^2_\mu}
\end{equation}
and $R$ is a bounded operator in $L^2 _\mu$ defined by
\begin{equation}
\label{R} R f :=\frac{1}{2\pi}\int_{\dR^2}\Big(-\ln\big|\cdot -
y\big| + C_{\rm E}+ \ln 2\Big)f(y)d\mu(y)\,;
\end{equation}
$C_{\rm E}$ stands for the Euler-Mascheroni constant\footnote{This
constant can be computed as $C_{\rm E} =
\lim\limits_{n\rightarrow\infty}\Big(\sum_{k=1}^n\tfrac{1}{k}
-\ln(n)\Big)$.}, i.e.~$C_{\rm E} = 0.57721...$.
\end{prop}
\begin{proof}
To prove the statement we employ the following expansion of the
Macdonald function 
\begin{equation}
K_0(x) = -\ln(x/2) + C_{\rm E} + s(x),\qquad x\rightarrow 0+\,,
\end{equation}
where  $s(x) = O(x^2\ln(x))$, see \cite[Equation 9.6.13]{AS64}. In
view of \eqref{Qkernel} and the compactness of the support of
$\mu$ the operator $Q(-k^2)$ can be expanded into the sum of the
rank-one operator $-\ln(k)P$, the operator $R$ and  the remaining
operator $S(k)$ with the integral kernel $s(k|x-y|)$. Since
$Q(-k^2)$, $P$ and $R$ are bounded the operator $S(k)$ is bounded
as well.  Further, note that for sufficiently small $k > 0$
\[
\big|s(k|x-y|)\big|\le A_\mu k^2|\ln(k)|,\qquad x,y\in{\rm supp}\,\mu,
\]
with some constant $A_\mu > 0$, which depends on $\mu$. Thus by Schur criterion the operator $S(k)$ in $L^2_\mu$
with the integral kernel $s(k|x-y|)$ satisfies
\[
\|S(k)\| = O(k^2\ln k),\qquad k\rightarrow 0+,
\]
which completes  the proof.
\end{proof}
\begin{remark}
Similar decomposition of the function $Q(\cdot)$ is employed  in
\cite{CK11} for some other purposes in the case of Dirac measure
supported by a non-compact curve.
\end{remark}
In the next lemma we gather some useful properties of the
operator-valued function $Q(\cdot)$.
\begin{lem}
\label{lem:Q} Let the operator-valued function $Q(\cdot)$ be
defined as in \eqref{Q}. Then the following statements hold.
\begin{itemize}\setlength{\parskip}{1.2mm}
\item [\rm (i)]  $Q(-k^2) \ge 0$ for all $k > 0$.
\item [\rm (ii)] $Q(-k_1^2)\le Q(-k_2^2)$ for $k_1 \ge k_2$.
\item [\rm (iii)] For any $\varepsilon >0$ there exists sufficiently small $k > 0$
such that the spectrum $\sigma(Q(-k^2))$ decomposes into two
disjoint parts
\[
\sigma_0(Q(-k^2))\subset (0,\|R\| +\varepsilon)
\]
with $R$ as in \eqref{R} and
\[
\sigma_1(Q(-k^2)) = \{\gamma(k)\},
\]
where $\gamma(k)$ is the eigenvalue of $Q(-k^2)$ with multiplicity
one.
\item [\rm (iv)] The function $\gamma(\cdot)$ is continuous, strictly decaying, and $\gamma(k)\rightarrow +\infty$ as $k\rightarrow 0+$.
\end{itemize}
\end{lem}
\begin{proof}
The item (i) follows directly from the  non-negativity of the
Macdonald function and  the representation of the integral kernel
of $Q(-k^2)$ given by \eqref{Qkernel}.

 The Macdonald function is
monotonously decaying function of its argument, which yields the
statement of (ii).

Note that according to Proposition~\ref{prop:expansion} the
function $ k\mapsto T(k)$, $k>0$ defined by
\[
T(-k^2) := -\frac{2\pi}{\mu_{\rm T}\ln k}Q(-k^2)
\]
determines  a realization of the operator family considered in
Theorem~\ref{thm:perturbation} with $\cH = L^2_\mu$, $\varphi =
\tfrac{\mathbbm 1_\mu}{\sqrt{\mu_{\rm T}}}$ and $T_1 =
-\tfrac{2\pi}{\mu_{\rm T}}R$ with $R$ as in \eqref{R}. Thus for
sufficiently small $k > 0$ the spectrum of the operator $Q(-k^2)$
can be separated into two parts as claimed in (iii) and the
function $\gamma(\cdot)$ is continuous. In view of (ii) the
function $\gamma(\cdot)$ is non-increasing. Suppose that for some
$k_1 < k_2$ the condition $\gamma(k_1) = \gamma(k_2)$ holds, that
implies $\gamma(k) = c >0$ for $k\in[k_1,k_2]$. Hence, by
Proposition~\ref{prop:BS} we have
$[-k_1^2,-k_2^2]\subset\sigma_{\rm p} (H_{(1/c)\mu})$, which is a
contradiction, because the point spectrum of any
self-adjoint operator should be a countable set. This proves
strict decay of $\gamma(\cdot)$.
\end{proof}

\subsection{Properties of the $R_{\mu\,dx}(\cdot)$-function}
\label{ssec:prelim4} In this subsection we investigate some
properties of the operator-valued function $R_{\mu\,dx}(\cdot )$
defined by \eqref{Rmu}. The unitary Fourier transform
$\mathcal{F}\,:\, L^2 \to L^2$ is defined as the extension
by continuity of the integral transform
\[
(\cF f)(p) := \frac{1}{2\pi}\int_{\dR^2}e^{-ipx}f(x)dx,\qquad f\in L^2\cap L^1.
\]
It is well-known that $\cF$ can be further extended by continuity up to
the space $\mathcal{S}'$, cf.~\cite[Chapter~1.1.7]{AH91}. Without
a danger of confusion we keep the same notation $\mathcal{F}\,:\,
\mathcal{S}'\to \mathcal{S}'$ for this extension. In the following
we will use also the abbreviation $\mathcal{F}f=  \wh f$, $f\in
\mathcal{S}'$.
Applying again the standard results concerning the
Sobolev spaces, see~\cite[Chapter~1.2.6]{AH91}, we can write
\begin{equation}\label{eq-Sobolev}
H^k =\{f \in \mathcal{S}' \,:\, \wh f (p){(p^2+1)^{k/2}} \in L^2
\}\,,
\end{equation}
where  the norm $\|\cdot \|_k$ in $H^k$ is defined by $\|f\|_k =
\| \wh f (p){(p^2+1)^{k/2}} \|$.
We define the functional $\varphi\mu$ for $\varphi\in L^2_\mu$ as
\[
(\varphi\mu)(f) := \int_{\dR^2} (J_\mu f)(x)\ov{\varphi(x)}d\mu(x),\qquad f\in H^1,
\]
with $J_\mu$ as in Subsection~\ref{ssec:prelim1}.
Let us show that $\varphi\mu\in H^{-1}$. Indeed for any $f\in H^1$ we get
\[
|(\varphi\mu)(f)| \le
\int_{\dR^2}|(J_\mu f)(x)||\varphi(x)|d\mu(x) \le
\|J_\mu f\|_{L^2_\mu}\|\varphi\|_{L^2_\mu} \le
C\|f\|_1\|\varphi\|_{L^2_\mu}
\]
with some constant $C >0 $, where we applied H\"older inequality in between and used that the embedding $J_\mu$ of $H^1$ into $L^2_\mu$ is continuous. We have shown that the functional $\varphi\mu$ is continuous on $H^1$
and hence $\varphi\mu\in H^{-1}$.
 Further, we define
\begin{equation}
\label{whvarphi} \wh\varphi(p) := (\cF(\varphi\mu))(p) = \frac{1}{2 \pi}
\int_{\dR^2}e^{-ipx}\varphi(x)d\mu(x),\qquad
\varphi \in L^2_\mu.
\end{equation}
In the next lemma we explore basic properties of the above
transform.
\begin{lem}
\label{lem:whphi} Let $\mu$ be a compactly supported positive
finite Radon measure on $\dR^2$ from the generalized Kato class.
Then for any $\varphi\in L^2_\mu$ its Fourier transform $\wh
\varphi$ given by \eqref{whvarphi} is a  bounded and Lipschitz
continuous function.
\end{lem}
\begin{proof}
Let $\varphi \in L^2_\mu $. Since the measure $\mu$ is finite the
inclusion $L^2_\mu\subset L^1_\mu$ holds. The boundedness of
$\wh\varphi$ follows from the estimate
\[
\|\wh\varphi\|_{L^\infty } \le
\tfrac{1}{2\pi}\|\varphi\|_{L^1_\mu} <\infty\,.
\]
It remains to show that $\wh\varphi$ is Lipschitz continuous. Let
us choose arbitrary $p_1,p_2\in\dR^2$. Applying (\ref{whvarphi})
we obtain
\begin{equation}
\label{Lip1} |\wh\varphi(p_1) -\wh\varphi(p_2)|\le
\frac{1}{2\pi} \int_{\dR^2}|e^{-ip_1x} -
e^{-ip_2x}|\cdot|\varphi(x)|d\mu(x).
\end{equation}
Using the fact that the function $\dR\ni t\mapsto e^{it}$ is
Lipschitz continuous we estimate
\begin{equation}
\label{Lip2} |e^{-ip_1x} - e^{-ip_2x}| = |1 - e^{-i(p_2 - p_1)x}|
\le L|x||p_2 -p_1|
\end{equation}
with some constant $L > 0$. Plugging \eqref{Lip2} into
\eqref{Lip1} and using compactness of $\mu$ we get
\[
|\wh\varphi(p_1) -\wh\varphi(p_2)|\le L'|p_2-p_1|
\]
with some constant $L' > 0$.
\end{proof}

\begin{remark}\label{re-extension}
{\rm Using the representation (\ref{eq-Sobolev}) of the Sobolev
spaces we can extend operator $R(-k^2)$ to a larger  space. To
derive this extension we apply
\begin{equation}\label{eq-RFour}
  R(-k^2) = \mathcal{F}^{-1}\frac{1}{|p|^2+k^2}\mathcal{F}\,:\,L^2 \to L^2\,,
\end{equation}
cf.~\cite{AH91}. Operator  $\frac{1}{|p|^2+k^2}\mathcal{F}$ is
bounded as the map acting from $H^{-1}$ to $L^2$ and, consequently,
it can be extended by continuity to the whole space $H^{-1}$. This
means that $R(-k^2)$ admits the analogous extension. Note
that $R_{\mu \, dx}(-k^2)\varphi$ with $\varphi\in L^2_\mu$ can be identified with the extension of $R(-k^2)$ defined above
applied to $\varphi \mu \in H^{-1}$.
 }
\end{remark}

In the next lemma we provide the Fourier representation of
$R_{\mu\, dx}(-k^2)$.
\begin{lem}
\label{lem:Rmu} Let $\mu$ be a compactly supported positive finite
Radon measure on $\dR^2$ from the generalized Kato class. The operator
$R_{\mu\, dx}(-k^2)\colon L^2_\mu\rightarrow L^2$ defined by
\eqref{Rmu} admits the representation
\begin{equation}\label{eq-Fourier}
R_{\mu\,dx}(-k^2)\varphi= \cF^{-1}\frac{\wh
\varphi(p)}{|p|^2+k^2}\,, \qquad \varphi\in L^2_\mu\,,
\end{equation}
where $\wh \varphi$ is given by \eqref{whvarphi} and $\cF^{-1}$ is
the inverse Fourier transform on $\dR^2$.
\end{lem}
\begin{proof} Combining the statements of
Remark~\ref{re-extension} and (\ref{whvarphi}) we get the claim.
\end{proof}

Having in mind later purpose we investigate in the next
proposition  the properties of $R_{\mu\,dx}(-k^2)$ as
$k\rightarrow 0+$.
\begin{prop}
\label{prop:Rmu} Let $\mu$ be a compactly supported positive
finite Radon measure on $\dR^2$ from the generalized Kato class.
Let the operator-valued function $R_{\mu\, dx}(-k^2)\colon
L^2_\mu\rightarrow L^2$ be as in \eqref{Rmu}. Then for any
$\varphi \in L^2_\mu$ the following asymptotics holds
\[
k^2\|R_{\mu\,dx}(-k^2)\varphi\|^2_{L^2} = \pi|\wh\varphi(0)|^2 +
O(\sqrt{k}),\qquad k\rightarrow 0+,
\]
where $\wh\varphi$ is the transform of $\varphi$ defined by
\eqref{whvarphi}.
\end{prop}
\begin{proof}
Let $\varphi\in L^2_\mu$ and $k >0$. Using Lemma~\ref{lem:Rmu} and
applying the fact that $\cF^{-1}$ is unitary in $L^2$ we obtain
\begin{equation} \label{eq-norm0}
k^2\|R_{\mu\, dx } (-k^2)\varphi \|^2
  =k^2\int_{\dR^2}
   \frac{|\widehat{\varphi}(p)|^2}{(|p|^2
   +k^2)^2}dp = \int_{\dR^2}\frac{|\widehat{\varphi}(kt)|^2}{(|t|^2
   +1)^2}dt\,.
\end{equation}
For given $\varepsilon >0$ we disjoin the last integral in
\eqref{eq-norm0} onto regions
\[
\mathcal{B}_{k}=\{t\in\dR^2\colon |t|<\tfrac {1}{\sqrt{k}}
\}\quad\text{and}\quad\mathcal{B}^{\rm c}_{k}=\dR^2 \setminus
\ov{\mathcal{B}_{k}}.
\]
Using boundedness of $\wh\varphi$ we obtain that
\begin{equation}\label{eq-norm2}
\begin{split}
\int_{\mathcal{B}^c_{ k} }
 \frac{|\widehat{\varphi}(k t)|^2}{(|t|^2 +1 )^2}
  dt& \le C\int_{\mathcal{B}^c_{ k} }
 \frac{1}{(|t|^2 +1 )^2}dt \\
 &= C'\int_{\frac{1}{\sqrt{k}}}^{+\infty}\frac{r}{(r^2+1)^2}dr  = O(k),\qquad k\rightarrow 0+.
 \end{split}
\end{equation}
Using boundedness and continuity of $\wh\varphi$, and applying mean-value theorem we arrive at
\begin{equation}\label{eq-norm-a}
\int_{\mathcal{B}_{k} }
 \frac{|\widehat{\varphi}(kt)|^2}
 {(|t|^2 +1)^2}dt =
 |\wh\varphi(\theta)|^2
 \int_{\mathcal{B}_{k} }
 \frac{dt} {(|t|^2 +1)^2},
\end{equation}
where $\theta\in\dR^2$ and $|\theta| \le \sqrt{k}$.
Applying the asymptotic behaviour
\[
\int_{\mathcal{B}_{k} }
 \frac{dt}
 {(|t|^2 +1)^2} =
 \int_{\dR^2}\frac{dt}{(|t|^2+1)^2} + O(k)
  = \pi + O(k),\qquad k\rightarrow 0+,
\]
to the formula \eqref{eq-norm-a} we obtain
\[
\int_{\mathcal{B}_{k} }
 \frac{|\widehat{\varphi}(kt)|^2}
 {(|t|^2 +1)^2}dt = \pi|\wh\varphi(\theta)|^2 + O(k),\qquad k\rightarrow 0+.
\]
Lipschitz continuity of $\wh\varphi$
combined with the above formula, \eqref{eq-norm0}, \eqref{eq-norm2} and $|\theta|\le \sqrt{k}$ imply that
\[
k^2 \|R_{\mu\, dx } (-k^2)\varphi  \|^2_{L^2} =
\pi|\wh\varphi(0)|^2 + O(\sqrt{k}),\qquad k\rightarrow 0+,
\]
and the claim is proven.
\end{proof}

\section{Weakly coupled bound state}
\label{sec-weakly}

In Subsection~\ref{ssec:main1} we show that for sufficiently small
coupling constant $\alpha
>0$ the discrete spectrum of
the self-adjoint operator $H_{\alpha\mu}$ consists of exactly one
negative eigenvalue of multiplicity one and we compute the
asymptotics of this eigenvalue as $\alpha\rightarrow 0+$.
Moreover, in Subsection~\ref{ssec:main2} we compute the
asymptotics of the corresponding eigenfunction in the same limit.

\subsection{Asymptotics of weakly coupled bound state}
\label{ssec:main1} In this subsection we compute the asymptotcs of
weakly coupled bound state. The  technique we employ here is
slightly different than the one applied in \cite{Si76}. As a
benefit it  allows to include also regular potentials with
stronger singularities.
\begin{thm}
\label{thm:existence} Let $\mu$ be a compactly supported positive
finite Radon measure on $\dR^2$ from the generalized Kato class.
Let the self-adjoint operator $H_{\alpha\mu}$ be as in
Definition~\ref{dfn:operator}. Then for all sufficiently small
$\alpha >0$ the condition
\[
\sharp\sigma_{\rm d}(H_{\alpha\mu}) = 1
\]
holds and the corresponding unique eigenvalue $\lambda(\alpha) <
0$ satisfies
\begin{equation}\label{eq-limitev}
\lambda(\alpha) \rightarrow 0-\quad \mathrm{for} \quad
\alpha\rightarrow 0+.
\end{equation}
\end{thm}
\begin{proof}
We rely on the Birman-Schwinger principle from
Proposition~\ref{prop:BS}. In order to recover the eigenvalues of
$H_{\alpha \mu }$ we will investigate the following condition
$1\in\sigma_{\rm p}(\alpha Q(-k^2))$. Let $\sigma_i(Q(-k^2))$,
$i=0,1$ be as in Lemma~\ref{lem:Q}\,(iii). The possibility
$1/\alpha\in\sigma_{ 0}(Q(-k^2))$ for $k > 0$ small enough is
excluded due to Lemma~\ref{lem:Q}\,(iii). On the other hand,
$1/\alpha\in\sigma_{1}(Q(-k^2))$ is equivalent to the equation
\[
\gamma(k) = 1/\alpha,
\]
which in view of Lemma~\ref{lem:Q}\,(iv) has exactly one solution
$k(\alpha )$ for $\alpha >0 $ small enough and moreover $k(\cdot
)$ satisfies
\[
k(\alpha) \rightarrow 0+, \qquad \alpha\rightarrow 0+.
\]
Consequently, $\lambda(\alpha) = -k(\alpha)^2$ gives the unique
negative simple eigenvalue of $H_{\alpha\mu}$ and the limiting
property \eqref{eq-limitev} holds.
\end{proof}
Our next aim is to derive asymptotics of $\lambda
(\alpha )$  for $\alpha \rightarrow 0+$.

\begin{thm}
\label{thm:main} Let $\mu$ be a compactly supported positive
finite Radon measure on $\dR^2$ from the generalized Kato class,
and let $H_{\alpha\mu}$ be the self-adjoint operator as in
Definition~\ref{dfn:operator}. Then the eigenvalue of $H_{\alpha
\mu}$ admits the following asymptotics
\begin{equation}\label{eq-evexp}
\lambda(\alpha) = -\big(C_\mu + o(1)\big)
\exp\Big(-\tfrac{4\pi}{\alpha \mu_{\rm T}}\Big)\,,\quad
\alpha\rightarrow 0+\,,
\end{equation}
with
\begin{equation}
\label{Cmu} C_\mu = \exp\Big(\frac{4\pi}{\mu_{\rm T}^2} (R
{\mathbbm 1}_\mu, {\mathbbm  1}_\mu)_{L^2_\mu }\Big),
\end{equation}
where $R$ is defined in \eqref{R}.
\end{thm}
\begin{proof} Let us consider the operator-valued function
\begin{equation}
\label{def:T} T(k) := -\frac{2\pi }{\mu_{\rm T} \ln k}
Q(-k^2),\qquad k
>0,
\end{equation}
where $Q(\cdot)$  is defined by \eqref{Q}. Comparing the expansion
from  Proposition~\ref{prop:expansion} and the definition
\eqref{def:T} one can see that the operator-valued function
$T(\cdot)$ reflects the structure assumed in
Theorem~\ref{thm:perturbation}; precisely $\cH = L^2_\mu$ and
\[
T_0 = 
\varphi_\mu \big(\cdot, \varphi_\mu 
\big)_{L^2_\mu}\,, \qquad \label{T1} T_1 = -\frac{2\pi}{\mu_{\rm
T}} R\,,
\]
where $\varphi \equiv \varphi_\mu
:=\frac{{\mathbbm{1}}_\mu}{\sqrt{\mu_{\rm T} }}$. Therefore, for
sufficiently small $k > 0$ the spectrum of $T(k)$ can be separated
into two disjoint parts: $\sigma_0(k)$ located in the neighborhood
of $0$ and $\sigma_1(k)$ consisting of exactly one simple
eigenvalue $\omega(k)$ located in the neighborhood of $1$ and
admitting  the asymptotic expansion
\[
\omega(k) = 1 +\tfrac{1}{\ln k} (T_1\varphi_\mu,\varphi_\mu
)_{L^2_\mu} + O\big(\tfrac{1}{\ln^2 k}\big), \qquad k\rightarrow
0+.
\]
Applying the definition of $T_1$ to the last expansion, we arrive
at
\begin{equation}
\label{omega} \omega(k) = 1 - \tfrac{2\pi}{\mu_{\rm T} ^2\ln k}(R
{\mathbbm  1}_\mu, {\mathbbm  1}_\mu)_{L^2_\mu} +
O\big(\tfrac{1}{\ln^2k}\big),\qquad k\rightarrow 0+.
\end{equation}
Suppose that $\alpha >0$ is sufficiently small, so that $\sharp
\sigma_{\rm d}(H_{\alpha\mu}) =
1$,~cf.~Theorem~\ref{thm:existence}. Let $\lambda(\alpha) =
-k^2(\alpha)$ standardly denote the corresponding unique
eigenvalue of $H_{\alpha\mu}$ which in view of
Theorem~\ref{thm:existence} converges as
 $\lambda(\alpha)\rightarrow 0-$  for $\alpha\rightarrow 0+$.
Combining the Birman-Schwinger principle together with the
definition of $T(\cdot)$ we obtain the following condition
\[
-\frac{1}{2\pi} \alpha\mu_{\rm T} \omega(k(\alpha)) \ln k(\alpha)
=1
\]
for the value $k(\alpha)$. Applying to the above equation the
asymptotic expansion of $\omega(\cdot)$ given by \eqref{omega} we
get
\[
-\frac{\alpha\mu_{\rm T} \ln k(\alpha)}{2\pi} +
\frac{\alpha}{\mu_{\rm T} }(R {\mathbbm  1}_\mu, {\mathbbm
1}_\mu)_{L^2_\mu } + O\big(\tfrac{\alpha}{\ln k(\alpha)}\big) =
1,\qquad\alpha\rightarrow 0+.
\]
The latter  is equivalent to
\begin{equation}\label{eq-lnconv}
\ln k(\alpha) =  -\frac{2\pi}{\alpha \mu_{\rm T} } +
\frac{2\pi}{\mu_{\rm T} ^2}(R {\mathbbm  1}_\mu, {\mathbbm
1}_\mu)_{L^2_\mu  } + o(1),\qquad \alpha\rightarrow 0+,
\end{equation}
which yields
\[
\lambda(\alpha) =  -k(\alpha)^2 = -\big(C_\mu
 + o(1)\big)e^{ -\frac{4\pi}{\alpha \mu_{\rm T} }},\qquad \alpha\rightarrow 0+,
\]
with $C_\mu$ as in \eqref{Cmu}.
\end{proof}

\begin{example}
\rm{ We will test the above theorem  on a special model. Namely,
let $\mu$ be defined via a Dirac measure supported on a circle
$C_r$ of radius $r$; precisely
\begin{equation}
\label{muc} \mu (\Omega ) =l (\Omega \cap C_r)\,,
\end{equation}
where $l(\cdot )$ is the one-dimensional measure defined by the
length of the arc. This example was already studied in
\cite{ET04}, where the authors compute negative spectrum of
$H_{\alpha\mu}$ (with $\mu$ as above) using separation of
variables.

In order to recover the asymptotic behavior of the eigenvalue of
$H_{\alpha \mu }$ with $\alpha $ small and  $\mu$ defined by
\eqref{muc}, we will compute the constant $C_\mu$ given
in \eqref{Cmu}. According to \cite[Lemma 3.2]{KV12} we obtain
\begin{equation}
\label{eq0} (Q(-k^2 )\mathbbm{1}_\mu ,  \mathbbm{1}_\mu)_{L^2_\mu}
=2\pi r^2\int_0^\infty\frac{|J_0(y)|^2y}{(kr)^2 + y^2}dy,
\end{equation}
where $J_0(\cdot)$ is the  Bessel function of order $0$. Applying
\cite[Equation 6.535]{GR} in the above formula we arrive at
\begin{equation}
\label{eq1} (Q(-k^2 )\mathbbm{1}_\mu ,
\mathbbm{1}_\mu)_{L^2_\mu}= 2\pi r^2 I_{0}(k r) K_0 (k r).
\end{equation}
Using the asymptotic expansions \cite[9.6.12, 9.6.13]{AS64}
\[
\begin{split}
I_0(x) &= 1 + O(x^2),\qquad x\rightarrow 0+,\\
K_0(x) &= \big(-\ln(x/2) + C_{\rm E}\big) + O(x^2\ln x),\qquad
x\rightarrow 0+,
\end{split}
\]
of $I_0 (\cdot )$ and $K_0 (\cdot )$ in the neighbourhood of zero,
we obtain
\begin{equation}
\label{eq2} I_0 (k r)K_0 (kr)= -\ln \frac{k r}{2} +C_{\rm E}
+O(k),\qquad k\to 0+.
\end{equation}
Combining equations \eqref{eq1} and \eqref{eq2} we get
\begin{equation}\label{eq-Q}
(Q(-k^2 )\mathbbm{1}_\mu ,  \mathbbm{1}_\mu)_{L^2_\mu}= 2\pi
r^2\Big( -\ln \frac{k r}{2} +C_{\rm E} + O(k ) \Big),\quad
k\rightarrow 0+.
\end{equation}
The decomposition stated in Proposition~\ref{prop-decomposition}
yields
\begin{equation}\label{eq-R}
(R {\mathbbm 1}_\mu , {\mathbbm 1}_\mu )_{L^2_\mu }= (Q (-k^2 )
{\mathbbm 1 }_\mu ,{\mathbbm 1}_\mu )_{L^2_\mu}+2\pi r^2 \ln k +O
(k^2 \ln k ),\quad k\rightarrow 0+.
\end{equation}
In fact, the left hand side of \eqref{eq-R} does not depend on
$k$. Consequently, inserting \eqref{eq-Q} into \eqref{eq-R} and
taking the limit $k\to 0+$ we get
\[
(R \mathbbm{1}_\mu , \mathbbm{1}_\mu )_{L^2_\mu }= 2\pi r^2 \Big(
-\ln \frac{r}{2} +C_{\rm E}  \Big)\,.
\]
In view of  \eqref{Cmu} this implies that $C_\mu = \frac{4}{r^2
}\exp(2C_{\rm E})$ and finally
\[
\lambda (\alpha )=-\frac{4}{r^2} \mathrm{e}^{2C_{\rm E} }
\mathrm{e}^{-\frac{2}{\alpha r}}(1+o(1)),\qquad \alpha\rightarrow
0+,
\]
which is fully consistent with a result of \cite[Subsection
2.1]{ET04} and furthermore refines that result.}
\end{example}
\begin{remark}
Following the line of \cite{BEKS94} one can introduce a sign
changing weight in $\gamma\in L^\infty(\dR^2)$ and consider more
general operators defined via quadratic forms
\[
\frq_{\alpha\gamma\mu}[f] := \|\nabla f\|^2_{{\mathbb L}^2} -
\alpha\int_{\dR^2}\gamma(x)|f(x)|^2d\mu(x),\quad \dom
\frq_{\alpha\gamma\mu} = H^1.
\]
In this case one can get the asymptotics similar to
\eqref{eq-evexp} with $\gamma$ involved. Instead of $\mu(\dR^2)$
in the exponent there will be $I := \int_{\dR^2}\gamma(x) d\mu(x)
> 0$. The asymptotics could be different if $I = 0$. This case
requires special analysis.
\end{remark}

\subsection{Asymptotics of the eigenfunction corresponding to the weakly coupled bound state}
\label{ssec:main2}
As we have shown in the previous section the operator
$H_{\alpha\mu}$ has exactly one negative eigenvalue for
sufficiently small $\alpha > 0$. The aim of this section is to
recover the asymptotic behaviour of the corresponding
eigenfunction in the limit $\alpha \to 0 +$.
\begin{thm} Let $\mu$ be a compactly supported positive finite Radon measure on $\dR^2$ from
the generalized Kato class. Let $\lambda (\alpha)$ be the unique
eigenvalue of $H_{\alpha \mu }$ for $\alpha \to 0+$ and
$-k_\alpha ^2 =\lambda (\alpha )$. Then the corresponding
eigenfunction has the form
\[
f_\alpha(\cdot)  =   \frac{k_\alpha}{2\pi} \int_{\dR^2} K_0
(k_\alpha| \cdot - y|) d\mu (y) + O\Big( \frac{1}{\ln k_\alpha}
\Big),\qquad\alpha\rightarrow 0+,
\]
where the error term is understood in the sense of $L^2$-norm;
moreover the $L^2$-norm of $f_\alpha$ has non-zero finite limit as
$\alpha\rightarrow 0+$.
\end{thm}
\begin{proof}
In the proof of this theorem we rely on Proposition~\ref{prop:BS}.
For non-trivial $\phi_\alpha \in \ker(I-\alpha Q(-k_\alpha ^2 ))$
the function
\begin{equation}\label{eq-normout}
g_\alpha(\cdot) := R_{\mu\, dx } (-k_\alpha^2)\phi_\alpha =
\frac{1}{2\pi}\int_{\dR^2} K_0(k_\alpha|\cdot-y|)\phi_\alpha (y)
d\mu(y),
\end{equation}
reproduces the eigenfunction of $H_{\alpha \mu}$. Similarly as in
the proof of Theorem~\ref{thm:main} we conclude that $\phi_\alpha$
is an eigenfunction of the operator
\[
T(k_\alpha) = -\frac{2\pi}{\mu_{\rm T}\ln(k_\alpha)}Q(-k^2_\alpha)
\]
corresponding to the eigenvalue $-\tfrac{2\pi}{\mu_{\rm
T}\ln(k_\alpha)\alpha}$. Recall that the family $\dR_+\ni k\mapsto
T(k)$ is a realization of the operator family considered in
Theorem~\ref{thm:perturbation} with $\cH = L^2_\mu$, $\varphi :=
\tfrac{\mathbbm{1}_\mu}{\sqrt{\mu_{\rm T}}}$, $T_0 :=
\varphi(\cdot,\varphi)$, $T_1 := -\tfrac{2\pi}{\mu_{\rm T}}R$ and
$R$ as in \eqref{R}. Hence, by
Theorem~\ref{thm:perturbation}\,(iii) we obtain that $\phi_\alpha$
can be chosen in the form
\begin{equation}\label{eq-phi}
\phi_\alpha  = \mathbbm{1}_\mu +\varsigma_\alpha,
\qquad\text{where}\quad \|\varsigma_\alpha \|_{L^2_\mu }=
O\Big(\tfrac{1}{\ln k_\alpha }\Big)\quad\text{as}\quad
\alpha\rightarrow 0+.
\end{equation}
By Proposition~\ref{prop:Rmu} we obtain
\begin{equation}
\label{kRmu}
k^2_\alpha \|R_{\mu\,dx} (-k^2_\alpha)\mathbbm{1}_\mu \|^2 =
\frac{\mu_{\rm T}^2}{4\pi}
 + O(\sqrt{k_\alpha}),\qquad \alpha\rightarrow 0+,
\end{equation}
where we used that $\wh{\mathbbm{1}}_\mu(0) =
\tfrac{1}{2\pi}\mu_{\rm T}$.
H\"older inequality
yields
\begin{equation}
\label{kRmu2} |\wh\varsigma_\alpha(0)| \le
\tfrac{1}{2\pi}\|\varsigma_\alpha\|_{L^1_\mu} \le \tfrac{1}{2\pi}
\|\varsigma_\alpha\|_{L^2_\mu}\sqrt{\mu_{\rm T}}.
\end{equation}
Hence, using Proposition~\ref{prop:Rmu}, \eqref{eq-phi} and \eqref{kRmu2}  we get
\begin{equation}
\label{kRmu1} k_\alpha^2 \|R_{\mu\, dx } (-k^2_\alpha
)\varsigma_\alpha  \|^2 = O\Big( \tfrac{1}{\ln^2 k_\alpha
}\Big)\qquad \alpha\rightarrow 0+,
\end{equation}
According to \eqref{eq-normout}, \eqref{eq-phi},
\eqref{kRmu} and \eqref{kRmu1}
\[
f_\alpha := k_\alpha R_{\mu\,dx}\phi_\alpha
\]
is an eigenfunction of $H_{\alpha\mu}$ and satisfies
\[
\|f_\alpha\|_{L^2}\rightarrow \frac{\mu_{\rm
T}}{2\sqrt{\pi}},\qquad\text{as}\quad \alpha\rightarrow 0+,
\]
moreover
\[
f_\alpha(\cdot) = \frac{k_\alpha}{2\pi}
\int_{\dR^2}K_0(k_\alpha|\cdot-y|)d\mu(y) + O\Big(\frac{1}{\ln
k_\alpha}\Big),\qquad \alpha\rightarrow 0+,
\]
holds, and the claim is proven.
\end{proof}
\subsection{Concluding remarks.} The asymptotics of the unique
eigenvalue  as well as the corresponding eigenfunction were proved
for the compactly supported measure $\mu$.
The assumption of the compactness
was essential, for example, for the decomposition
\eqref{eq-Qdecomposition} which was the fundamental tool for the
proof of Theorem~\ref{thm:main}. It seems that the most natural
way is to apply approaching of non-compactly supported
measure $\mu$ by an
appropriate sequence $\mu_n$ of compactly supported
measures. However, we face the problem that the error term $o(1)$ appearing in
\eqref{eq-evexp} is not, generally, uniform with respect to $n$.

\end{document}